\documentclass[12pt]{article}

\usepackage{amsfonts}
\usepackage{amsmath}
\usepackage{amssymb}
\usepackage{color}
\usepackage{amsthm}
\usepackage{hyperref}
\usepackage{makecell}
\usepackage{graphicx}
\usepackage{multirow}

\theoremstyle{plain}
\newtheorem{theorem}{Theorem}

\newtheorem{proposition}[theorem]{Proposition}

\theoremstyle{definition}

\theoremstyle{remark}
\newtheorem{remark}[theorem]{Remark}

\title{Geometric approximation for the number of returns of a transient Markov chain to its origin} 

\author{Fraser Daly\footnote{Department of Actuarial Mathematics and Statistics, and the Maxwell Institute for Mathematical Sciences, Heriot--Watt University, Edinburgh EH14 4AS, UK. E-mail: F.Daly@hw.ac.uk}\, and Seva Shneer\footnote{Department of Actuarial Mathematics and Statistics, and the Maxwell Institute for Mathematical Sciences, Heriot--Watt University, Edinburgh EH14 4AS, UK. E-mail: V.Shneer@hw.ac.uk}} 

\date{\today}

\begin{document}

\maketitle

\noindent{\bf Abstract} 
Given a discrete-time, transient Markov chain, we establish explicit total variation error bounds in the approximation of the number of visits to its starting state within the first $n$ time steps by a geometric distribution. Our error bounds are expressed in terms of the heat kernel of the Markov chain. As applications, we consider Galton--Watson processes with a geometric offspring distribution, and random walks on the one- and three-dimensional integer lattices, regular trees, supercritical percolation clusters, and groups with polynomial growth.

\vspace{12pt}

\noindent{\bf Key words and phrases: }Markov chain; return time; Galton--Watson process; random walk; Stein's method; heat kernel

\vspace{12pt}

\noindent{\bf MSC 2020 subject classification: }60F05 (Primary); 60G42; 60G50; 62E17 (Secondary)

\section{Introduction and main result}

Let $\{X_i:i=0,1,\ldots\}$ be a discrete-time Markov chain with $X_0=x^\star$, a transient state. In this note we give an explicit error bound in the geometric approximation of the number of returns to the state $x^\star$ within the first $n$ time steps. That is, we bound the total variation distance between the random variable $N_n$ defined by
\begin{equation}\label{eq:Ndef}
N_n=\sum_{i=1}^nI(X_i=x^\star)\,,
\end{equation}
where $I(\cdot)$ denotes an indicator function, and a geometric distribution with a suitable parameter. The total variation distance is defined by
\[
d_\text{TV}(N,G)=\sup_{A\subseteq\mathbb{N}}|\mathbb{P}(N\in A)-\mathbb{P}(G\in A)|\,,
\]
for random variables $N$ and $G$ taking values in $\mathbb{N}=\{0,1,\ldots\}$. We say that $G$ has a geometric distribution with parameter $p\in(0,1]$, written $G\sim\text{Geom}(p)$, if $\mathbb{P}(G=j)=p(1-p)^{j}$ for $j\in\mathbb{N}$. By transience of $x^\star$ we have that $\lim_{n\to\infty}\mathbb{P}(N_n=0)>0$, and this will be the parameter we choose for our approximating geometric distribution. We emphasise that in many applications our results give bounds valid for a fixed $n$, not only rates of convergence in total variation distance. 

In this work, we are motivated by the work of Pek\"oz and R\"ollin \cite{pr11b}, who give error bounds in Wasserstein distance for exponential approximation of the number of returns of a Markov chain to its origin. Their proofs make use of a coupling argument and Stein's method, which we also employ. These techniques have also been applied to a variety of other approximation problems related to Markov chains, including Poisson \cite[Section 8.5]{bhj92} and compound Poisson \cite{e99} approximations for the number of visits to rare sets, and geometric \cite{p96} and compound geometric \cite{d10} approximations for hitting times.

We illustrate our main result with applications to Galton--Watson processes with a geometric offspring distribution, simple random walks on the one- and three-dimensional integer lattices, and random walks on regular trees, supercritical percolation clusters and groups with polynomial growth. In a random walk setting, $N_n$ has applications to particle motion, and to understanding search efficiency on a fixed or random network. In Galton--Watson models for contact tracing in an epidemic context, beginning with a single individual, $N_n$ is the number of secondary infection waves originating from single individuals.  

Our applications are discussed in Sections \ref{sec:GW}--\ref{sec:GG}, before which we establish our main result, a geometric approximation bound in terms of the heat kernel of the underlying Markov chain.
\begin{theorem}\label{thm:main}
Let $\{X_i:i=0,1,\ldots\}$ be a Markov chain with transient initial state $X_0=x^\star$, and $N_n$ be as defined in \eqref{eq:Ndef}.  
Let $q=\lim_{n\to\infty}\mathbb{P}(N_n=0)>0$ and $H\sim\text{Geom}(q)$. Then
\[
d_\text{TV}(N_n,H)\leq q\left\lceil\frac{1-q}{q}\right\rceil\sum_{i=n+1}^\infty\mathbb{P}(X_i=x^\star)\,,
\]
where $\lceil\cdot\rceil$ is the ceiling function.
\end{theorem}
\begin{proof}
We begin by following the approach of Pek\"oz \cite{p96}, employing Stein's method for geometric approximation by coupling the random variables $N_n+1$ and $N_n|N_n>0$. In doing so we will make use of
\[
T=\inf\{i\geq1:X_i=x^\star\}\,,
\]
the first return time to $x^\star$, which we note may be infinite. 
Since the event $\{N_n>0\}$ is equal to the event $\{T\leq n\}$, we proceed by sampling $T^\star$ from the distribution of $T|T\leq n$, and then $\{X^\star_i:i\geq1\}$ from the distribution of $\{X_i:i\geq1\}$ conditional on the event that $X_{T^\star}=x^\star$. Letting $\eta_i=I(X_i^\star=x^\star)$ and using the regenerative property of the Markov chain, we may then construct the required random variables by writing 
\begin{equation}\label{eq:coupling}
N_n+1\stackrel{d}{=}\sum_{i=T^\star}^{n+T^\star}\eta_i \qquad\text{and}\qquad N_n|N_n>0\stackrel{d}{=}\sum_{i=1}^n\eta_i\,,
\end{equation}
where `$\stackrel{d}{=}$' denotes equality in distribution. With $\eta_i=0$ for $i=1,\ldots,T^\star-1$ by construction, the difference between the random variables in \eqref{eq:coupling} is $\sum_{i=n+1}^{n+T^\star}\eta_i$, which is almost surely non-negative. 

We use this coupling to establish a bound on $d_\text{TV}(N_n,G)$, where $G\sim\text{Geom}(p)$ and $p=\mathbb{P}(N_n=0)$. To that end we follow \cite{p96} and for a given $A\subseteq\mathbb{N}$ let $f_A:\mathbb{N}\to\mathbb{R}$ be such that $f_A(0)=0$ and
\begin{equation}\label{eq:SteinEq}
(1-p)f_A(j+1)-f_A(j)=I(j\in A)-\mathbb{P}(G\in A)\,,
\end{equation}
for $j\geq0$. Then 
\begin{equation}\label{eq:SteinFactor}
\sup_{j\in\mathbb{N}}|f_A(j+1)-f_A(j)|\leq1
\end{equation}
for any $A$ (see Lemma 2.5 of \cite{prr13}), and using \eqref{eq:coupling}--\eqref{eq:SteinFactor} we may write
\begin{align*}
|\mathbb{P}(N_n\in A)-\mathbb{P}(G\in A)|&=(1-p)|\mathbb{E}[f_A(N_n+1)]-\mathbb{E}[f_A(N_n)|N_n>0]|\leq(1-p)\mathbb{E}\sum_{i=n+1}^{n+T^\star}\eta_i\\
&=(1-p)\left(\mathbb{E}[N_n+1]-\mathbb{E}[N_n|N_n>0]\right)
=1-p(1+\mathbb{E}N_n)\,,
\end{align*}
giving $d_\text{TV}(N_n,G)\leq1-p(1+\mathbb{E}N_n)$. To complete the proof we combine this with the triangle inequality:
\[
d_\text{TV}(N_n,H)\leq d_\text{TV}(N_n,G)+d_\text{TV}(G,H)\,.
\]
Letting $r(n)=p(n)-q$, by Proposition 2.5 of \cite{aj06} we have $d_\text{TV}(G,H)\leq\left\lceil\frac{1-q}{q}\right\rceil r(n)$, since $p(n)$ is decreasing in $n$ and so $r(n)\geq0$. Then
\begin{align*}
d_\text{TV}(N_n,H)&\leq1-(q+r(n))(1+\mathbb{E}N_n)+\left\lceil\frac{1-q}{q}\right\rceil r(n)
\leq1-q(1+\mathbb{E}N_n)+\left(\left\lceil\frac{1-q}{q}\right\rceil-1\right)r(n)\\
&=q\sum_{i=n+1}^\infty\mathbb{P}(X_i=x^\star)+\left(\left\lceil\frac{1-q}{q}\right\rceil-1\right)r(n)\,,
\end{align*}
noting for the final equality that the number of returns to $x^\star$ over an infinite time horizon has a geometric distribution with parameter $q$ and mean $\frac{1-q}{q}=\sum_{i=1}^\infty\mathbb{P}(X_i=x^\star)$, and writing
\[
\mathbb{E}N_n=\sum_{i=1}^n\mathbb{P}(X_i=x^\star)=\frac{1-q}{q}-\sum_{i=n+1}^\infty\mathbb{P}(X_i=x^\star)\,.
\]

Letting $a_i=\mathbb{P}(X_i=x^\star)$, it remains only to show that $r(n)\leq q\sum_{i=n+1}^\infty a_i$. To that end, we use a standard renewal-theoretic argument to write $\mathbb{E}N_n=\mathbb{P}(T\leq n)+\sum_{j=1}^{n-1}\mathbb{P}(T=j)\mathbb{E}N_{n-j}$.
Combining this with the expression $\mathbb{E}N_n=\sum_{i=1}^n a_i$, we have that
\[
\sum_{i=1}^na_i=1-\mathbb{P}(T>n)+\sum_{i=1}^{n-1}a_i\mathbb{P}(T\leq n-i)\,,
\]
from which it follows that $\mathbb{P}(T>n)=1-a_n-\sum_{i=1}^{n-1}a_i\mathbb{P}(T>n-i)$,
or, equivalently,
\[
p(n)=1-a_n-\sum_{i=1}^{n-1}a_ip(n-i)=1-\sum_{i=1}^na_ip(n-i)\,,
\]
defining $p(0)=1$. Since $p(n)\geq q$ for each $n\geq0$, we have
\[
r(n)=1-q-\sum_{i=1}^na_ip(n-i)\leq1-q-q\sum_{i=1}^na_i=1-q\left(1+\sum_{i=1}^\infty a_i\right)+q\sum_{i=n+1}^\infty a_i\,.
\]
Finally, recalling that $q(1+\sum_{i=1}^\infty a_i)=1$, we obtain $r(n)\leq q\sum_{i=n+1}^\infty a_i$, completing the proof. 
\end{proof}

\begin{remark}
For comparison with Theorem \ref{thm:main}, we note that we easily obtain an exact expression for the Wasserstein distance $d_\text{W}(N_n,H)$ between $N_n$ and $H$ by observing that $N_n$ is stochastically smaller than $H$ and so $d_\text{W}(N_n,H)=\mathbb{E}[H-N_n]=\sum_{i=n+1}^\infty\mathbb{P}(X_i=x^\star)$. This Wasserstein distance also serves as an upper bound on the total variation distance, but our upper bound in Theorem \ref{thm:main} has an additional factor of $q\left\lceil\frac{1-q}{q}\right\rceil\leq1$. Although this factor is close to 1 in many cases (such as when $q$ is close to 0 or 1), it allows us to improve the explicit numerical value of our upper bound in many of the applications we consider below. For example, this improves the bound of Proposition \ref{cor:GWcritical} by a factor of $6/\pi^2$, and that of Proposition \ref{prop:RWabsorb} by a factor of $1/2$.
\end{remark}
\begin{remark}
In the continuous-time setting, we may similarly evaluate the Wasserstein distance between time spent at the origin and an exponential random variable. Let $\{Y_t:t\geq0\}$ be a continuous-time Markov chain with transient initial state $Y_0=y^\star$ and transition rate matrix with $(i,j)$th entry $\lambda_{i,j}$ for $i\not=j$. After arrival in a state $i$, the process spends an amount of time in that state which is exponentially distributed with rate $\lambda_i=\sum_{j\not=i}\lambda_{i,j}$. Let $T_n$ denote the total amount of time spent in state $y^\star$ up to the time immediately before the ($n+1$)th transition. Letting $\{X_i:i=0,1,\ldots\}$ be the embedded discrete-time Markov chain such that $X_0=y^\star$ and $X_i$ is equal to the state of $\{Y_t:t\geq0\}$ immediately following its $i$th transition, we may write $T_n=\sum_{i=1}^{N_n+1}\tau_i$, where the $\tau_i$ are IID and exponentially distributed with rate $\lambda_{y^\star}$, and $N_n$ is as in \eqref{eq:Ndef}. 
Let $E$ be an exponential random variable with rate $q\lambda_{y^\star}$ for $q$ as in Theorem \ref{thm:main} for the embedded discrete-time chain. Noting that $E$ is equal in distribution to $\sum_{i=1}^{H+1}\tau_i$, where $H\sim\text{Geom}(q)$, and that $N_n$ is stochastically smaller than $H$, we have that $T_n$ is stochastically smaller than $E$ and so
$d_\text{W}(T_n,E)=\mathbb{E}\left[E-T_n\right]=\lambda_{y^\star}^{-1}\sum_{i=n+1}^\infty\mathbb{P}(X_i=y^\star)$. Establishing related bounds in other metrics in this setting is left as a question for future research, as is the question of whether total variation bounds can be established here without reduction of the problem to the embedded discrete-time chain. The equilibrium coupling techniques employed in conjunction with Stein's method to establish Wasserstein distance bounds in \cite{pr11b} in the continuous-time case are related to, but distinct from, the coupling arguments we used in our Theorem \ref{thm:main}, and may be relevant to these questions.
\end{remark}

\section{Galton--Watson processes with geometric offspring} \label{sec:GW}

We now consider applications of Theorem \ref{thm:main} to the generation sizes of a Galton--Watson process with geometric offspring distribution. That is, beginning with a single individual in generation 0, each individual alive at time $i$ independently produces a geometrically distributed number of offspring; the total offspring produced by individuals in the $i$th generation constitutes the $(i+1)$th generation. We give explicit geometric approximation results for the number of the first $n$ generations to contain exactly one individual, as defined by \eqref{eq:Ndef} with $x^\star=1$, beginning with the critical case (where the offspring distribution has unit mean), then discussing the supercritical case (in which this mean is larger than one). Stein's method has also been used by Pek\"oz, R\"ollin and Ross \cite{prr13} to give geometric approximation bounds for the generation size of a critical Galton--Watson process conditioned on non-extinction; see also \cite{cjp24,pr11b,pr11a} for related exponential approximation results. 

\subsection{The critical case}

Let $X_i$ be the size of the $i$th generation of a critical Galton--Watson process whose offspring distribution is geometric with parameter $1/2$, starting at generation 0 with a single ancestor. Our geometric approximation result in this case builds upon recent work of \cite{wz25} who showed in their Corollary 3 that $N_n$ converges in distribution to a geometric random variable with parameter $6/\pi^2$ as $n\to\infty$, so that, in the notation of Theorem \ref{thm:main}, $q=6/\pi^2$ here and $\left\lceil\frac{1-q}{q}\right\rceil=1$. 
It is well known that 
\[
\mathbb{E}[s^{X_i}]=1-\frac{1-s}{1+i(1-s)}
\]
for $i\geq1$ and $s\in[0,1]$; see Section I.4 of \cite{an72}. It follows that $\mathbb{P}(X_i=1)=\frac{1}{(i+1)^2}$. We note that
\[
\sum_{i=n+1}^\infty\mathbb{P}(X_i=1)=\sum_{i=n+2}^\infty\frac{1}{i^2}\leq\int_{n+1}^\infty\frac{\text{d}x}{x^2}=\frac{1}{(n+1)}\,.
\]
Theorem \ref{thm:main} then immediately gives the following proposition.
\begin{proposition}\label{cor:GWcritical}
Let $\{X_i:i=0,1,\ldots\}$ be a critical Galton--Watson process with geometric offspring distribution and $X_0=1$. Let $H\sim\text{Geom}(6/\pi^2)$. Then
\[
d_\text{TV}(N_n,H)\leq\frac{6}{\pi^2(n+1)}\,.
\]
\end{proposition}

\subsection{The supercritical case}

We now let $X_i$ be the size of the $i$th generation of a Galton--Watson process whose offspring distribution is geometric with parameter $\alpha<1/2$, starting at generation 0 with a single ancestor. Let $\mu$ denote the mean of this offspring distribution, $\mu=(1-\alpha)/\alpha>1$. Since the probability-generating function of our offspring distribution is $f(s)=\frac{\alpha}{1-(1-\alpha)s}$, the non-unit solution of $f(s)=s$ is given by $s=\mu^{-1}$. Hence, the probability-generating function of the size of the $i$th generation is
\[
\mathbb{E}[s^{X_i}]=1-\mu^i\left(\frac{\mu-1}{\mu^{i+1}-1}\right)+\frac{\mu^{i}(\mu-1)^2s}{(\mu^{i+1}-1)^2-\mu(\mu^i-1)(\mu^{i+1}-1)s}\,,
\]
for $i\geq1$ and $s\in[0,1]$; again, see Section I.4 of \cite{an72}. It follows that $\mathbb{P}(X_i=1)=\mu^i\left(\frac{\mu-1}{\mu^{i+1}-1}\right)^2$ and
\[
\sum_{i=n+1}^\infty\mathbb{P}(X_i=1)=(\mu-1)^2\sum_{i=n+1}^\infty\frac{\mu^i}{(\mu^{i+1}-1)^2}\leq(\mu-1)^2\int_n^\infty\frac{\mu^x}{(\mu^{x+1}-1)^2}\,\text{d}x=\frac{(\mu-1)^2}{\mu\log(\mu)(\mu^{n+1}-1)}
\]
for $n=0,1,\ldots$. From Theorem \ref{thm:main} we then have that 
\[
d_\text{TV}(N_n,H)\leq\frac{\lceil\theta\rceil(\mu-1)^2}{(1+\theta)\mu\log(\mu)(\mu^{n+1}-1)}\,,
\]
where
\begin{equation}\label{eq:theta}
\theta=\frac{1-q}{q}=\lim_{n\to\infty}\mathbb{E}N_n=(\mu-1)^2\sum_{i=1}^\infty\frac{\mu^i}{(\mu^{i+1}-1)^2}\leq\frac{(\mu-1)}{\mu\log(\mu)}<1
\end{equation}
for all $\mu>1$. Hence, $\left\lceil\frac{1-q}{q}\right\rceil=\lceil\theta\rceil=1$. Noting that each term in the infinite series in \eqref{eq:theta} is non-negative, we may take the first term of that series as a lower bound for $\theta$. This gives $\theta\geq\frac{\mu(\mu-1)^2}{(\mu^2-1)^2}=\frac{\mu}{(\mu+1)^2}$, and we obtain the following:
\begin{proposition}
Let $\{X_i:i=0,1,\ldots\}$ be a supercritical Galton--Watson process with geometric offspring distribution with mean $\mu>1$, and $X_0=1$. Let $H\sim\text{Geom}(q)$ and $q$ be as above. Then
\[
d_\text{TV}(N_n,H)\leq\frac{(\mu^2-1)^2}{\mu\log(\mu)([\mu+1]^2+\mu)(\mu^{n+1}-1)}\,.
\]
\end{proposition}

\section{Simple random walks on the integers} \label{sec:RW}

In this section we apply our Theorem \ref{thm:main} to a selection of simple random walks, specifically transient simple random walks on $\mathbb{Z}$ and  $\mathbb{Z}^3$. The selection of applications presented here is intended to be illustrative rather then exhaustive. As in Section \ref{sec:GW}, the bounds we give here are explicit with moderate constants, at the cost of a slightly lengthier exposition than would be needed to obtain only rates of convergence. In the setting of the simple random walk on $\mathbb{Z}$, D\"obler \cite{d15} has previously given explicit error bounds in the Kolmogorov and Wasserstein distances for the approximation of the number of returns to the origin by a half-normal distribution, and in the two-dimensional setting Pek\"oz and R\"ollin \cite{pr11b} give exponential approximation bounds in Wasserstein distance for the number of returns of a general random walk to the origin.

\subsection{Transient random walk on $\mathbb{Z}$}

Let $\{X_i:i=0,1,\ldots\}$ be a simple random walk on $\mathbb{Z}$ with $X_0=0$ and in which the process takes a step in the positive direction with probability $\alpha\not=1/2$ or a step in the negative direction with probability $1-\alpha$ at each time, independently of all other steps. Let $N_n$ be as in \eqref{eq:Ndef} with $x^\star=0$, and assume without loss of generality that $\alpha>1/2$. It is well known that the probability the process never returns to 0 is $q=2\alpha-1$. It is clear that $\mathbb{P}(X_{2i+1}=0)=0$ for $i\in\mathbb{N}$, and
\[
\mathbb{P}(X_{2i}=0)=[\alpha(1-\alpha)]^i\binom{2i}{i}\leq\frac{[4\alpha(1-\alpha)]^i}{\sqrt{\pi i}}\,,
\]
since $\binom{2i}{i}\leq\frac{4^i}{\sqrt{\pi i}}$. Hence,
\[
\sum_{i=n+1}^\infty\mathbb{P}(X_i=0)=\sum_{i=\left\lceil\frac{n+1}{2}\right\rceil}^\infty\mathbb{P}(X_{2i}=0)\leq\sum_{i=\left\lceil\frac{n+1}{2}\right\rceil}^\infty\frac{[4\alpha(1-\alpha)]^i}{\sqrt{\pi i}}\leq\frac{1}{1-4\alpha(1-\alpha)}\sqrt{\frac{2[4\alpha(1-\alpha)]^n}{\pi n}}\,,
\]
since $\left\lceil\frac{n+1}{2}\right\rceil\geq\frac{n}{2}$. We immediately obtain the following from Theorem \ref{thm:main}:
\begin{proposition}
Let $\{X_i:i=0,1,\ldots\}$ be the above simple random walk on $\mathbb{Z}$ with parameter $\alpha>1/2$ and $X_0=0$. Let $H\sim\text{Geom}(2\alpha-1)$. Then
\[
d_\text{TV}(N_n,H)\leq\frac{2\alpha-1}{1-4\alpha(1-\alpha)}\left\lceil\frac{2(1-\alpha)}{2\alpha-1}\right\rceil\sqrt{\frac{2[4\alpha(1-\alpha)]^n}{\pi n}}\,.
\]
\end{proposition}

\subsection{Symmetric random walk on $\mathbb{N}$ with absorption}

We consider a classical gambler's ruin problem. Let $\{X_i:i=0,1,\ldots\}$ be a simple, symmetric random walk on the non-negative integers, with $X_0=1$ and with an absorbing barrier at zero. That is, at each time the process takes a step in either the positive or negative direction, each with probability $1/2$, until it first reaches zero. Let $N_n$ be as in \eqref{eq:Ndef} with $x^\star=1$. Since the random walk without absorption is recurrent, the probability that our process never returns to its starting state is $q=1/2$. It is clear that $\mathbb{P}(X_{2i+1}=1)=0$ for $i\in\mathbb{N}$, and that, since the number of paths in our random walk which return to state $1$ in $2i$ steps without hitting the absorbing barrier is given by the Catalan number $\frac{1}{i+1}\binom{2i}{i}$, we have
\[
\mathbb{P}(X_{2i}=1)=\frac{1}{(i+1)4^i}\binom{2i}{i}\leq\frac{1}{(i+1)\sqrt{\pi i}}\leq\frac{1}{\sqrt{\pi}i^{3/2}}\,,
\]
since $\binom{2i}{i}\leq\frac{4^i}{\sqrt{\pi i}}$. Hence,
\[
\sum_{i=n+1}^\infty\mathbb{P}(X_i=1)\leq\frac{1}{\sqrt{\pi}}
\sum_{i=\left\lceil\frac{n+1}{2}\right\rceil}^\infty\frac{1}{i^{3/2}}\leq\frac{1}{\sqrt{\pi}}\int_{\left\lceil\frac{n+1}{2}\right\rceil-1}^\infty\frac{\text{d}x}{x^{3/2}}\leq\frac{2^{3/2}}{\sqrt{\pi(n-2)}}
\]
for $n\geq3$, since $\left\lceil\frac{n+1}{2}\right\rceil\geq n/2$. We immediately obtain the following proposition from Theorem \ref{thm:main}:
\begin{proposition}\label{prop:RWabsorb}
Let $\{X_i:i=0,1,\ldots\}$ be the above symmetric random walk on $\mathbb{N}$ with an absorbing barrier at zero and $X_0=1$. Let $H\sim\text{Geom(1/2)}$. Then, for $n\geq3$,
\[
d_\text{TV}(N_n,H)\leq\sqrt{\frac{2}{\pi(n-2)}}\,.
\]
\end{proposition}

\subsection{Symmetric random walk on $\mathbb{Z}^3$}

Let $\{X_i:i=0,1,\ldots\}$ be a simple symmetric random walk on $\mathbb{Z}^3$, beginning at the origin $O$, and $N_n$ be as in \eqref{eq:Ndef} with $x^\star=O$. We let $n\geq5$ to avoid inessential complications in the analysis. It is well known that the probability that the random walk returns to the origin at some time is $0.340537330\ldots$ so that, in the notation of Theorem \ref{thm:main}, $q=1-0.3405\ldots$ and $\left\lceil\frac{1-q}{q}\right\rceil=1$; see page 103 of \cite{s76}. 
A combinatorial argument gives that $\mathbb{P}(X_{2i+1}=O)=0$ for $i\in\mathbb{N}$, and
\[
\mathbb{P}(X_{2i}=O)=\left(\frac{1}{6}\right)^{2i}\sum_{j+k+\ell=i}\binom{2i}{2j,2k,2\ell}\binom{2j}{j}\binom{2k}{k}\binom{2\ell}{\ell}
=\left(\frac{1}{6}\right)^{2i}\binom{2i}{i}\sum_{j+k+\ell=i}\left(\frac{i!}{j!k!\ell!}\right)^2\,,
\]
where $\binom{2i}{2j,2k,2\ell}$ is the trinomial coefficient. The well-known inequalities 
\[
\exp\left\{\frac{1}{12i+1}\right\}\leq\frac{i!}{\sqrt{2\pi i}}\left(\frac{\mathrm{e}}{i}\right)^i\leq\exp\left\{\frac{1}{12i}\right\}
\]
give, for $j+k+\ell=i$,
\begin{equation}\label{eq:stirling}
\frac{ii!}{3^ij!k!\ell!}\leq\frac{1}{2\pi}\times\frac{i^{3/2}}{\sqrt{jk\ell}}\times\frac{i^i}{3^ij^jk^k\ell^\ell}\times\frac{\exp\left\{\frac{1}{12i}\right\}}{\exp\left\{\frac{1}{12j+1}+\frac{1}{12k+1}+\frac{1}{12\ell+1}\right\}}\,.
\end{equation}
We consider separately the three possible values of $i\bmod3$, beginning with the case where $i$ is a multiple of 3. In this case, $\frac{i!}{j!k!\ell!}$ with $j+k+\ell=i$ is maximised by $j=k=\ell=i/3$, and \eqref{eq:stirling} gives
\begin{equation}\label{eq:trinomial}
\frac{i!}{j!k!\ell!}\leq\frac{c3^i}{2\pi i}\,,
\end{equation}
for some constant $c$. We may clearly take $c=3^{3/2}$ here, but will need to allow for a larger value of this constant in other cases. Considering these other cases, it is again straightforward to show that (i) if $i\equiv1\bmod3$ then $\frac{i!}{j!k!\ell!}$ is maximised by $j=k=\frac{i-1}{3}$ and $\ell=\frac{i+2}{3}$, in which case \eqref{eq:trinomial} again holds with $c=\sqrt{2}$ for $i\geq2$, and (ii) if $i\equiv2\bmod3$ then $\frac{i!}{j!k!\ell!}$ is maximised by $j=k=\frac{i+1}{3}$ and $\ell=\frac{i-2}{3}$, in which case \eqref{eq:trinomial} again holds with $c=7$ for $i\geq3$. Combining these observations with the fact that $\sum_{j+k+\ell=i}\frac{i!}{j!k!\ell!}=3^i$, we have that, for $i\geq3$,
\begin{equation}\label{eq:3Dprob}
\mathbb{P}(X_{2i}=O)\leq\frac{7}{2\pi i}\left(\frac{1}{2}\right)^{2i}\binom{2i}{i}\leq\frac{7}{2(\pi i)^{3/2}}\,,
\end{equation}
since $\binom{2i}{i}\leq\frac{4^i}{\sqrt{\pi i}}$ and where we have chosen the constant 7 so that this inequality holds for all values of $i\bmod3$.
We thus have
\[
\sum_{i=n+1}^\infty\mathbb{P}(X_i=O)=\sum_{i=\left\lceil\frac{n+1}{2}\right\rceil}^\infty\mathbb{P}(X_{2i}=O)
\leq\frac{7}{2\pi^{3/2}}\sum_{i=\left\lceil\frac{n+1}{2}\right\rceil}^\infty\frac{1}{i^{3/2}}
\leq\frac{7}{2\pi^{3/2}}\int_{\left\lceil\frac{n+1}{2}\right\rceil-1}^\infty\frac{\text{d}x}{x^{3/2}}
\leq7\sqrt{\frac{2}{\pi^3n}}\,,
\]
where $n\geq5$ ensures that we can apply \eqref{eq:3Dprob}. Theorem \ref{thm:main} then immediately gives the following:
\begin{proposition}
Let $\{X_i:i=0,1,\ldots\}$ be a simple symmetric random walk on $\mathbb{Z}^3$ with $X_0=O$. Let $H\sim\text{Geom}(q)$ with $q$ as above. Then, for $n\geq5$,
\[
d_\text{TV}(N_n,H)\leq7q\sqrt{\frac{2}{\pi^3n}}\,.
\]
\end{proposition}

\section{Random walks on other infinite graphs and groups}\label{sec:GG}

In each of our preceding applications we have obtained error bounds with explicit constants. In this section we consider further applications to random walks in which error bounds are available via appropriate heat kernel estimates, but may be somewhat less explicit than those in Sections \ref{sec:GW} and \ref{sec:RW}. Our selection of applications is again illustrative rather than exhaustive; see, for example, \cite{hs93,w00} for further models to which our results may be applied.

\subsection{Random walk on a regular tree}

Let $\{X_i:i=0,1,\ldots\}$ be a simple random walk on an underlying infinite, regular tree of degree $d\geq3$, with $X_0=O$. By Lemma 1.9 and Theorem 11.1 of \cite{w00} we have that
\[
\sum_{i=n+1}^\infty\mathbb{P}(X_i=O)\leq\sum_{i=n+1}^\infty\left(\frac{2\sqrt{d-1}}{d}\right)^i=\frac{2\sqrt{d-1}}{d-2\sqrt{d-1}}\left(\frac{2\sqrt{d-1}}{d}\right)^n\,.
\]
Using the same bound for $\mathbb{P}(X_i=O)$ we have that $q$, the parameter of the geometric limit of $N_n$ as in \eqref{eq:Ndef}, satisfies $\frac{1-q}{q}\leq\frac{2\sqrt{d-1}}{d-2\sqrt{d-1}}$, so that $q\geq1-\frac{2\sqrt{d-1}}{d}$. Hence, Theorem \ref{thm:main} gives the following:
\begin{proposition}
For $\{X_i:i=0,1,\ldots\}$ a simple random walk on a regular tree of degree $d\geq3$,    
\[\
d_\text{TV}(N_n,H)\leq \frac{2q\sqrt{d-1}}{d-2\sqrt{d-1}}\left\lceil\frac{1-q}{q}\right\rceil\left(\frac{2\sqrt{d-1}}{d}\right)^n\,,
\]
for some $q\geq1-\frac{2\sqrt{d-1}}{d}$, where $H\sim\text{Geom}(q)$ and $N_n$ is as in \eqref{eq:Ndef}.
\end{proposition}

\subsection{Random walk on an infinite percolation cluster in $\mathbb{Z}^d$}

Throughout this section we let $C$ denote a positive constant, independent of $n$, whose value may vary from line to line. Consider the graph $(\mathbb{Z}^d,\mathcal{E}_d)$ for some $d\geq3$, where $\mathcal{E}_d=\{\{x,y\}:|x-y|=1\}$. For each $e\in\mathcal{E}_d$ we have an independent Bernoulli random variable $\omega(e)$ with mean $\alpha$. Edges $e$ with $\omega(e)=1$ are open. We assume that $\alpha>\alpha_c$, the critical probability for this graph, so that there is a unique infinite cluster of vertices $\mathcal{C}$ connected by paths of open edges, and without loss of generality we assume the origin $O\in\mathcal{C}$. For a fixed configuration $\omega\in\{0,1\}^{\mathcal{E}_d}$, we let  $\{X_i:i=0,1,\ldots\}$ be a simple random walk on $\mathcal{C}$ with $X_0=O$, where at each step the process moves along an open edge adjacent to its current position chosen uniformly at random from the set of such edges.
By Theorem 5.1 of \cite{bh09} we have that $\mathbb{P}(X_i=O)\leq Ci^{-d/2}$ for all suitably large $i$. Hence, for all $n$ large enough we have
\begin{equation}\label{eq:perc}
\sum_{i=n+1}^\infty\mathbb{P}(X_i=O)\leq C\sum_{i=n+1}^\infty i^{-d/2}\leq C\int_n^\infty x^{-d/2}\,\text{d}x=Cn^{1-d/2}\,,
\end{equation}
and Theorem \ref{thm:main} gives the following:
\begin{proposition}
Let $\{X_i:i=0,1,\ldots\}$ be a simple random walk on an infinite percolation cluster in $\mathbb{Z}^d$ for some $d\geq3$. Then there exist constants $C>0$ and $q\in(0,1]$ such that
\[
d_\text{TV}(N_n,H)\leq Cn^{1-d/2}\,,
\]
for $n$ large enough, where $H\sim\text{Geom}(q)$ and $N_n$ is as in \eqref{eq:Ndef}.
\end{proposition}

\subsection{Random walk on a group with polynomial growth}

Throughout this section we again let $C$ denote a positive constant, independent of $n$, whose value may vary from line to line. Let $\nu$ be a symmetric probability measure supported on a finite set which generates a group $\Gamma$. Let $\{X_i:i=0,1,\ldots\}$ be an irreducible and aperiodic random walk on $\Gamma$ with $X_0=O$ and transition probabilities $p(x,y)=\nu(x^{-1}y)$. We further assume that $\Gamma$ has polynomial growth, that is, the number of elements of $\Gamma$ with word length less than $n$ is of order $n^d$ for some $d$, and that $d\geq3$. Under these assumptions, $\mathbb{P}(X_i=O)\leq Ci^{-d/2}$ for all $i\geq1$; see page 674 of \cite{hs93}. Then \eqref{eq:perc} holds for all $n$, and Theorem \ref{thm:main} gives the following:
\begin{proposition}
Let $\{X_i:i=0,1,\ldots\}$ be a random walk on a group with polynomial growth rate $d\geq3$ as above. Then there exist constants $C>0$ and $q\in(0,1]$ such that
\[
d_\text{TV}(N_n,H)\leq Cn^{1-d/2}\,,
\]
where $H\sim\text{Geom}(q)$ and $N_n$ is as in \eqref{eq:Ndef}.
\end{proposition}


\begin{thebibliography}{99}

\bibitem{aj06} J.~A.~Adell and P.~Jodr\'a (2006). Exact Kolmogorov and total variation distances between some familiar discrete distributions. \emph{J. Inequal. Appl.} {\bf 2006}: 64307.

\bibitem{an72} K.~B.~Athreya and P.~E.~Ney (1972). \emph{Branching Processes}. Springer, Berlin.

\bibitem{bhj92} A.~D.~Barbour, L.~Holst and S.~Janson (1992). \emph{Poisson Approximation}. Oxford University Press, Oxford.

\bibitem{bh09} M.~T.~Barlow and B.~M.~Hambly (2009). Parabolic Harnack inequality and local limit theorem for percolation clusters. \emph{Electron. J. Probab.} {\bf 14}(1): 1--26.

\bibitem{cjp24} N.~Cardona-Tob\'{o}n, A.~Jaramillo and S.~Palau (2024). Rates on Yaglom’s limit for Galton--Watson processes in a varying environment. \emph{ALEA, Lat. Am. J. Probab. Math. Stat.} {\bf 21}(1): 1--23.

\bibitem{d10} F.~Daly (2010). Stein's method for compound geometric approximation. \emph{J. Appl. Probab.} {\bf 47}(1): 146--156.

\bibitem{d15} C.~D\"obler (2015). Stein’s method for the half-normal distribution with applications to limit theorems related to the simple symmetric random walk. \emph{ALEA, Lat. Am. J. Probab. Math. Stat.} {\bf 12}(1): 171--191.

\bibitem{e99} T.~Erhardsson (1999). Compound Poisson approximation for Markov chains using Stein's method. \emph{Ann. Probab.} {\bf 27}(1): 565--596.

\bibitem{hs93} W.~Hebisch and L.~Saloff-Coste (1993). Gaussian estimates for Markov chains and random walks on groups. \emph{Ann. Probab.} {\bf 21}(2): 673--709.

\bibitem{p96} E.~Pek\"oz (1996). Stein's method for geometric approximation. \emph{J. Appl. Probab.} {\bf 33}(3): 707--713.

\bibitem{pr11b} E.~A.~Pek\"oz and A.~R\"ollin (2011). Exponential approximation for the nearly critical Galton--Watson process and occupation times of Markov chains. \emph{Electron. J. Probab.} {\bf 16}:  1381--1393.

\bibitem{pr11a} E.~A.~Pek\"oz and A.~R\"ollin (2011). New rates for exponential approximation and the theorems of R\'{e}nyi and Yaglom. \emph{Ann. Probab.} {\bf 39}(2): 587--608.

\bibitem{prr13} E.~A.~Pek\"oz, A.~R\"ollin and N.~Ross (2013). Total variation error bounds for geometric approximation. \emph{Bernoulli} {\bf 19}(2): 610--632.

\bibitem{s76} F.~Spitzer (1976). \emph{Principles of Random Walk}, Second Edition. Springer, New York.

\bibitem{wz25} H.-M.~Wang and S.~Zhang (2025). Beyond Poisson approximation: Sums of Markovian Bernoulli variables with applications to Brownian motions and branching processes. Preprint, available at \texttt{https://arxiv.org/abs/2504.19404}.

\bibitem{w00} W.~Woess (2000). \emph{Random Walks on Infinite Graphs and Groups}. Cambridge Tracts in Mathematics 138, Cambridge Univ. Press, Cambridge.

\end{thebibliography}
\end{document}